\documentclass[12 pt]{article}%
\usepackage{amsmath, amsfonts, amsthm, color,latexsym}
\usepackage{amsmath}
\usepackage{amsfonts}
\usepackage{amssymb}
\usepackage{color, soul}
\usepackage[all]{xy}
\usepackage{graphicx}%
\providecommand{\U}[1]{\protect\rule{.1in}{.1in}}
\allowdisplaybreaks[4]
\newtheorem{theorem}{Theorem}[section]
\newtheorem{proposition}[theorem]{Proposition}

\newtheorem{example}[theorem]{Example}

\newtheorem{remark}[theorem]{Remark}

\newtheorem{lemma}[theorem]{Lemma}
\newtheorem{final remark}[theorem]{Final Remark}
\newtheorem{definition}[theorem]{Definition}
\begin{document}

\title{Generalized adjoints and applications to composition operators}
\author{Geraldo Botelho\thanks{Supported by CNPq Grant
305958/2014-3 and Fapemig Grant PPM-00450-17.}\,\, and  Leodan A. Torres\thanks{Supported by a CNPq scholarship.
}}
\date{}
\maketitle

\begin{abstract} We generalize the classical notion of adjoint of a linear operator and the Aron-Schottenloher notion of adjoint of a homogeneous polynomial. The general notion is shown to enjoy several properties enjoyed by the classical ones, nevertheless differences between the two theories are detected. The proofs of both positive and negative results are not simple adaptations of the linear cases, actually nonlinear arguments are often required. Applications of the generalized adjoints to Lindstr\"om-Schl\"uchtermann type theorems for composition operators are provided.
\end{abstract}

\section*{Introduction and background}

The adjoint (dual, conjugate or transpose) $u^* \colon F^* \longrightarrow E^*$ of a bounded linear operator $u \colon E \longrightarrow F$ between Banach spaces is a central notion in Linear Functional Analysis. Extending this notion to the nonlinear setting, Aron and Schottenloher \cite{aron} defined the adjoint $P^*$ of a continuous homogeneous polynomial $P$, which instantly became a basic tool in Nonlinear Functional Analysis and in Infinite Dimensional Holomorphy. Recent applications of the adjoint of a homogeneous polynomial can be found, e.g., in \cite{erhanpilar, silvia, pabloturco}.

Our purpose in this paper is to show that these adjoints are particular cases of a much more general notion, which we call generalized adjoints (cf. Definition \ref{def1}). In Section 1 we develop the first properties of these generalized adjoints, establishing, on the one hand, that many features of the classical theories are actually particular cases of the general theory. On the other hand, a crucial difference between the two theories is detected (cf. Proposition \ref{w.rcp} and Example \ref{exdif}), making clear that there is room for further research in the general theory.

Due to the strong nonlinear flavor of the general theory, it is expected that canonical linear arguments do not work in the general setting. Of course this is true, nonlinear arguments are required throughout the paper; but, as important as that, the general setting also discloses phenomena that could not be discovered in the classical theory (cf. Proposition \ref{delta.ida}), reinforcing the pertinence of future research in the subject.

In Section 2 we show how the generalized adjoints can be useful by  proving nonlinear versions of some linear results due to Lindstr\"om and Schl\"uchtermann \cite{lindstrom} on composition operators, which, in particular, recover the original results as particular cases.

By $E$ and $F$ we denote real or complex Banach spaces, $E^*$ denotes the topological dual of $E$, $B_E$ stands for the closed unit ball of $E$ and $J_E \colon E \longrightarrow E^{**}$ is the canonical embedding. The symbols ${\cal L}(E;F)$ and ${\cal P}(^mE;F)$ denote the Banach spaces of continuous linear operators and continuous $m$-homogeneous polynomials from $E$ to $F$, $m \in \mathbb{N}$, endowed with the usual sup norm. When $F$ is the scalar field $\mathbb{K} = \mathbb{C}$ or $\mathbb{R}$ we simply write ${\cal P}(^mE)$. The Aron-Schottenloher adjoint of a continuous $m$-homogeneous polynomial $P \in {\cal P}(^mE;F)$ is the following linear operator:
$$P^* \colon F^* \longrightarrow {\cal P}(^mE)~,~ P^*(y^*) = y^* \circ P. $$
By $\check{P}$ we denote the (unique) continuous symmetric $m$-linear operator from $E^m$ to $F$ that generates the polynomial $P \in {\cal P}(^mE;F)$. The following well known formula shall be used several times: for $x \in E$ and $m \in \mathbb{N}$,
\begin{equation} \label{formula}\|x\|^m=\sup_{q\in B_{\mathcal{P}(^m E)}} |q(x)|.
\end{equation}
For the general theory of (spaces of) homogeneous polynomials between Banach spaces we refer to \cite{livrodineen, mujicalivro}.
%\textcolor{red}{Ver depois...}

%\section{Introduction and background}

%\textcolor{red}{Aqui tem de ter toda a notacao, as definicoes dos simbolos, o minimo de sequence classes, etc, etc.}
%
%The letters $E, E_1, \ldots, E_n,F$ shall denote Banach spaces over $\mathbb{K} = \mathbb{R}$ or $\mathbb{C}$. The closed unit ball of $E$ is denoted by $B_E$ and its topological dual by $E'$. By $J_E$ we denote the canonical embedding of $E$ in $E''$. The symbol $BAN$ denote the class of all Banach spaces over $\mathbb{K}$. Given Banach spaces $E$ and $F$, the symbol $E \stackrel{1}{\hookrightarrow} F$ means that $E$ is a linear subspace of $F$ and $\|x\|_F \leq \|x\|_E$ for every $x \in E$. By ${\cal L}(E;F)$ we denote the Banach space of continuous linear operators $T \colon E \longrightarrow F$ endowed with the usual sup norm. The symbol  $T'$ denote the adjoint of an operator $T \in {\cal L}(E;F)$.  We denote by $p^*$ the conjugated index of $p$, that is, $1= 1/p + 1/{p^*}$.

\section{Generalized adjoints}

We start by defining the generalized adjoints.
%%In this section we introduce the generalized adjoints of linear operators and homogeneous polynomials and prove many properties which show that, much more than simply recovering the original notions as particular cases, in many aspects these generalized notions behave like the original ones. Nevertheless, an important difference between the classical and the new theories is detected .

\begin{definition}\label{def1}\rm Let $m,n,k$ be given natural numbers.  Given a continuous $m$-homogeneous
polynomial $P \in \mathcal{P}(^m E; F)$, define
$$\Delta^n_k P\colon \mathcal{P}(^k F)\longrightarrow\mathcal{P}(^{mnk} E)~,~ \Delta ^n_k P(q)(x)=q(P(x))^n.$$
\end{definition}
This concept recovers the classical adjoint of a linear operator and the Aron-Schottenloher adjoint of a homogeneous polynomial as follows:\\
\indent $\bullet$ For $u \in {\cal L}(E;F)$, $u^* = \Delta_1^1 u$.\\
\indent $\bullet$ For $P \in {\cal P}(^mE;F)$, $P^* = \Delta_1^1 P$.

\begin{proposition}\label{pripro}{\rm (a)} $\Delta^n_k P$ is a well defined continuous $n$-homogeneous polynomial, that is, $\Delta^n_k P \in {\cal P}(^n \mathcal{P}(^k F),\mathcal{P}(^{mnk} E)),$ and $\|\Delta^n_k P\| = \|P\|^{kn}$.\\
{\rm (b)} $\Delta_k^n$ is a continuous $kn$-homogeneous polynomial, that is, $$\Delta_k^n \in {\cal P}\left(^{kn}{\cal P}(^mE;F); {\cal P}(^n {\cal P}(^kF); {\cal P}(^{mkn}E)) \right),$$
and $\|\Delta_k^n\| = 1$.
\end{proposition}
This result generalizes the following facts: $u^* \in {\cal L}(E;F)$, $P^* \in {\cal L}(F^*; {\cal P}(^mE))$, $\|u^*\| = \|u\|$, $\|P^*\| = \|P\|$, and the correspondences $u \mapsto u^*$ and $P \mapsto P^*$ are norm 1 linear operators.
\begin{proof} (a) We just prove the norm equality. Calling on (\ref{formula}) for the first time,
\begin{align*}
\|\Delta^n_k P\|& =\sup_{\substack{\|q\|\leq1 }}\| \Delta^n_k P(q)\| =\sup_{\substack{\|q\|\leq1 }}\sup_{\substack{\|x\|\leq1 }}| \Delta^n_k P(q)(x)|=\sup_{\substack{\|q\|\leq1 }}\sup_{\substack{\|x\|\leq1 }}| q(P(x))|^ n\\
&=\sup_{\substack{\|x\|\leq1 }}\sup_{\substack{\|q\|\leq1 }} | q(P(x))|^n=\sup_{\substack{\|x\|\leq1 }}\|P(x)\| ^{kn}=\|P\|^{kn}.
\end{align*}
(b) Use (a) and note that the map $A \colon {\cal P}(^mE;F)^{kn} \longrightarrow {\cal P}(^n {\cal P}(^kF); {\cal P}(^{mkn}E))$ given by
$$A(P_1,\ldots,P_k,\ldots,P_{nk})(q)(x)=\check{q}(P_1(x),\ldots, P_k(x))\cdots\check{q}(P_{(n-1)k+1}(x),\ldots, P_{kn}(x)),$$
is a continuous $kn$-linear operator that generates $\Delta_k^n$.   %for all $P_1, \ldots P_k, \ldots, P_{nk}\in\mathcal{P}(^m E;F)$, $q\in\mathcal{P}(^k F)$, $x\in E$
\end{proof}

To describe the behavior of $\Delta_k^n$ with respect to the algebraic operations, we use the following lemma.

\begin{lemma}\label{prilema} Given a polynomial $R \in {\cal P}(^m E;F)$, there is a polynomial $W_{R} \in {\cal P}(^m(E \times E); F)$ such that:\\
{\rm (a)} $R(x + y) = R(x) + R(y) + W_{R}(x,y)$ for all $x,y \in E$.\\
{\rm (b)} $W_{R}= 0$ if and only if $m = 1$.
\end{lemma}

\begin{proof} Assume, wlog, that $R \neq 0$. For $m=1$ simply take $W_{R}=0$. For $m>1$, from \cite[Lemma 1.9]{livrodineen} we have  \begin{equation*}
R(x+y)=\displaystyle\sum_{j=0}^{m}\binom{m}{j}\check{R}(x^{(j)},y^{(m-j)})=R(x)+R(y)+
\displaystyle\sum_{j=1}^{m-1}\binom{m}{j}\check{R}(x^{(j)},y^{(m-j)}),
\end{equation*}
where $\check{R}(x^{(j)},y^{(m-j)}) = \check{R}(x, \stackrel{(j)}{\ldots},{x}, y, \stackrel{(m-j)}{\ldots}, y)$.
The first statement follows from the fact that the map
$$ W_{R}\colon  E\times E\longrightarrow F~,~W_{R}(x,y)= \displaystyle\sum_{j=1}^{m-1}\binom{m}{j}\check{R}(x^{(j)},y^{(m-j)}),$$
is a continuous $m$-homogeneous polynomial. Indeed, $W_R$ is generated by the continuous $m$-linear operator $A\colon (E\times E)^m \longrightarrow F$ given by
\begin{equation*}
A\left((x_1,y_1),(x_2,y_2),\ldots, (x_m,y_m)\right)=
\displaystyle\sum_{j=1}^{m-1}\binom{m}{j}\check{R}(x_1,x_2,\ldots,x_j,y_{j+1},y_{j+2},\ldots, y_m).
\end{equation*}
It is clear that $W_R = 0 $ if $m=1$. Now suppose that $W_R = 0$. In this case we have $R(x+y)=R(x)+R(y)$ for all $x, y\in E$. Let $x_0 \in E$ be such that $R(x_0) \neq 0$. Since $R$ is an $m$-homogeneous polynomial,
$2^m R(x_0) = R(2x_0)=2R(x_0)$, from which it follows that $m = 1$.
\end{proof}

The next result, which follows from Lemma \ref{prilema} and Proposition \ref{pripro}, generalizes the formulas $(u + \lambda v)^* = u^* +\lambda v^*$, $(P + \lambda Q)^* = P^* + \lambda Q^*$ and shows that the classical correspondences $u \mapsto u^*$ and  $P \mapsto P^*$ are the only ones that are linear.

\begin{proposition} {\rm (a)} There is a $kn$-homogeneous polynomial $W_{\Delta_k^n}$ from ${\cal P}(^mE;F)\times {\cal P}(^mE;F)$ to ${\cal P}(^n {\cal P}(^kF); {\cal P}(^{mkn}E))$ such that
$$\Delta_k^n(P+Q) = \Delta_k^nP + \Delta_k^n Q + W_{\Delta_k^n}(P,Q) $$
for all $P,Q \in {\cal P}(^mE;F)$. \\
{\rm (b)} $\Delta_k^n(P+Q) = \Delta_k^nP + \Delta_k^n Q$ for all $P,Q \in {\cal P}(^mE;F)$ if and only if $k = n = 1$.\\
{\rm (c)} $\Delta_k^n (\lambda P) = \lambda^{kn}\Delta_k^n P$ for all $\lambda \in \mathbb{K}$ and $P \in {\cal P}(^mE;F)$.
\end{proposition}

The correspondences $u \mapsto u^*$ and  $P \mapsto P^*$ are injective. To investigate the injectivity of the general correspondence $P \mapsto \Delta_n^k P$, since it is a homogeneous polynomial, we have first to recall when a homogeneous polynomial can be injective. The following is well known:

\begin{lemma} \label{seglema} If there exists an injective polynomial in ${\cal P}(^m E;F)$, then either $m = 1$ or $m$ is odd and $\mathbb{K} = \mathbb{R}$.
\end{lemma}

Bearing the lemma above in mind, the next result shows that the correspondence $P \mapsto \Delta_n^k P$ is injective whenever it can be injective.

\begin{proposition} If either $k = n = 1$ or $kn$ is odd and $\mathbb{K} = \mathbb{R}$, then the correspondence
$P \in {\cal P}(^mE;F) \mapsto \Delta_n^k P \in  {\cal P}(^n {\cal P}(^kF); {\cal P}(^{mkn}E)) $ is injective.
\end{proposition}

\begin{proof}
The case $k=n=1$ is well known (alternatively, it follows from Proposition \ref{pripro}). Let $kn$ be odd, $\mathbb{K}=\mathbb{R}$ and $P_1 \not= P_2$. Take  $x_0\in E$ such that $P_1(x_0)\not= P_2(x_0)$ and, by Hahn-Banach, let $y^*\in F^*$ be such that $y^*(P_1(x_0))\not= y^*(P_2(x_0))$. Therefore $(y^*)^k \in {\cal P}(^kF)$ and
$$(\Delta_{k}^{n}P_1)((y^*)^k)(x_0) =  y^*(P_1(x_0))^{kn}\not= y^*(P_2(x_0))^{kn}= (\Delta_{k}^{n}P_2)((y^*)^k)(x_0).$$
\end{proof}

Now we show that the formulas $(u \circ v)^* = v^* \circ u^*$ and $(u \circ P)^* = P^* \circ u^*$ are particular instances of a much more general formula.

\begin{proposition} Let $m,n,k,r,s \in \mathbb{N}$. If $P \in {\cal P}(^m E;F)$ and $Q \in {\cal P}(^r F;G)$, then
$$\Delta_k^{ns}(Q \circ P) = \Delta_{rnk}^sP \circ \Delta_k^n Q. $$
\end{proposition}
\begin{proof}
For $q\in\mathcal{P}(^k G)$ and $x\in E$,
\begin{align*}
(\Delta_{rnk}^{s}P\circ \Delta_{k}^{n}Q)(q)(x)&=(\Delta_{rnk}^{s}P)(\Delta_{k}^{n}Q(q))(x)= \left[\Delta_{k}^{n}Q(q) (P(x))\right]^s\\
& = q(Q(P(x)))^{ns}= q((Q\circ P)(x))^{ns}=\Delta_{k}^{ns}(Q\circ P)(q)(x).
\end{align*}
\end{proof}

If a linear operator $u$ is surjective (respectively, an isomorphism), then its adjoint $u^*$ is injective (respectively, an isomorphism). Now we give more general versions of these facts. The reader should keep in mind the restrictions given in Lemma \ref{seglema} for a homogeneous polynomial to be injective.

\begin{proposition} {\rm (a)} Let $P \in {\cal P}(^mE;F)$ be a surjective polynomial. If either $n = 1$ or $n$ is odd and $\mathbb{K} = \mathbb{R}$, then $\Delta_k^n P$ is injective for every $k \in \mathbb{N}$.\\
{\rm (b)} If $j \colon G\twoheadrightarrow E$ is a metric surjection, then $\Delta^1_k j$ is a metric injection for every $k\in\mathbb{N}$.\\
{\rm (c)} If $u \in {\cal L}(E;F)$ is an (isometric) isomorphism, then $\Delta_k^1 u$ is an (isometric) isomorphism and $(\Delta_k^1 u)^{-1} = \Delta_k^1 (u^{-1})$ for every $k \in \mathbb{N}$.
\end{proposition}

\begin{proof}
$(a)$ The case $n=1$ is easy and we omit it. In the case $n$ odd and $\mathbb{K}=\mathbb{R}$, $\Delta_{k}^{n}P$ is a continuous $n$-homogeneous polynomial between real Banach spaces. Let  $q_1$, $q_2 \in \mathcal{P}(^k F)$ be such that $\Delta_{k}^{n}P(q_1)=\Delta_{k}^{n}P(q_2)$. Then $q_1(P(x))^n = q_2(P(x))^n$, hence $q_1(P(x))= q_2(P(x))$ for every $x\in E$. Since $P$ is surjective we have $q_1= q_2$.\\
(b) Denoting by $\stackrel{\circ}{ B_E}$ the open unit ball of $E$, since $j(\stackrel{\circ}{ B_G}) =\stackrel{\circ}{ B_E}$  \cite[B.3.6]{pietschlivro},
$$\|\Delta^1_k j(q)\|=\|q\circ j\|=\sup_{\substack{ \|x\|<1 }} | q(j(x))| =\sup_{\substack{ \|y\|<1 }} | q(y)|=\|q\|$$
for every $q\in\mathcal{P}(^k E)$.\\
(c) Let  $k\in\mathbb{N}$. By (a) we know that $\Delta_{k}^{1}u$ is injective. %Now we show that $\Delta_{k}^{1}u$ is surjective, that is, for every $R\in\mathcal{P}(^k E)$, there exists $S\in\mathcal{P}(^k F)$ such that $\Delta_{k}^{1}u(S)=R$. In fact, if
Given $R\in\mathcal{P}(^k E)$, $S:=R \circ u^{-1} \in \mathcal{P}(^k F)$ and $$\Delta_{k}^{1}u(S)(x)=S(u(x))=(R\circ u^{-1})(u(x))=R(x)$$
for every $x \in E$, proving that $\Delta_{k}^{1}u$ is bijective. Since $\Delta_{k}^{1}u$ is a continuous linear operator between Banach spaces, it follows from the open mapping theorem that $\Delta_{k}^{1}u$ is an isomorphism. Now suppose that $u$ is an isometric isomorphism. Then $\Vert x \Vert\leq 1\Leftrightarrow \Vert u(x)\Vert \leq 1$, from which it follows that
$$
\Vert \Delta_{k}^{1}u(q)\Vert
=\sup_{\Vert x \Vert \leq 1} \vert q(u(x))\vert
=\sup_{\Vert u(x)\Vert \leq 1} \vert q(u(x))\vert
=\Vert q \Vert
$$
for every $q\in \mathcal{P}(^k F)$. Considering the chains of operators
 $$
\begin{array}{ccccccc}
E & \stackrel{u}{\longrightarrow} & F &  \stackrel{u^{-1}}{\longrightarrow} & E
&  \stackrel{u}{\longrightarrow} & F {\rm ~and}
\end{array}
 $$
 $$
\begin{array}{ccccccc}
\mathcal{P}(^k F)& \stackrel{\Delta^1_k u}{\longrightarrow} &\mathcal{P}(^{k}E) &  \stackrel{\Delta^1_{k}(u^{-1} )}{\longrightarrow} & \mathcal{P}(^{k}F)&  \stackrel{\Delta^1_{k}u}{\longrightarrow} & \mathcal{P}(^{k}E)
\end{array},
 $$
from
$$\Delta_{k}^{1}(u^{-1}\circ u)
=\Delta_{k}^{1} id_{E}
=id_{\mathcal{P}(^k E)}~\,\, {\rm and}~\,\, \Delta_{k}^{1}(u\circ u^{-1})
=\Delta_{k}^{1} id_{F}
=id_{\mathcal{P}(^k F)}$$
it follows that
$$\Delta_{k}^{1}u \circ \Delta_{k}^{1}(u^{-1})
=id_{\mathcal{P}(^k E)} {\rm \,\,~and\,\,~}\Delta_{k}^{1}(u^{-1}) \circ \Delta_{k}^{1}u
=id_{\mathcal{P}(^k F)}. $$
%$$
%\Delta_{k}^{1}(u^{-1}\circ u)
%=\Delta_{k}^{1} id_{E}
%=id_{\mathcal{P}(^k E)}
%\Rightarrow \Delta_{k}^{1}u \circ \Delta_{k}^{1}(u^{-1})
%=id_{\mathcal{P}(^k E)}
%$$
%
%$$
%\Delta_{k}^{1}(u\circ u^{-1})
%=\Delta_{k}^{1} id_{F}
%=id_{\mathcal{P}(^k F)}
%\Rightarrow \Delta_{k}^{1}(u^{-1}) \circ \Delta_{k}^{1}u
%=id_{\mathcal{P}(^k F)}.
%$$
\end{proof}

Let $u \in {\cal L}(E;F)$ and let $J_E \colon E \longrightarrow E^{**}$ be the canonical embedding. Our next purpose is to show that the well known commutative diagram
\begin{equation*}
\begin{gathered}
\xymatrix@C2pt@R15pt{
           E\ar@/_/[dd]_*{J_E} \ar[rrrrrrrrrr]^*{u} & & & & &  & & & & & F \ar@/^/[dd]^*{J_F}&\\
&  &  &  &  & & & & &  & & & & &  &\\
         E^{**}\ar[rrrrrrrrrr]^*{u^{**}}  & & & & & & & & & & F^{**}
}
\end{gathered}
\end{equation*}
holds true at a very high level of generality. First we need the following generalization of the canonical embedding $J_E$.

\begin{lemma}\label{terlema} For $m,n \in \mathbb{N}$, the map
$$J_E^{m,n} \colon E \longrightarrow {\cal P}(^m{\cal P}(^nE))~,~J_E^{m,n}(x)(q) = q(x)^m, $$
is a continuous $mn$-homogeneous polynomial and $\|J_E^{m,n}(x)\| = \|x\|^{mn}$ for every $x \in E$.
\end{lemma}
It is clear that $J_E^{1,1} = J_E$.

\begin{proof} For $x_1, \ldots, x_{mn} \in E$, the map $B \colon {\cal P}(^nE)^m \longrightarrow \mathbb{K}$,
$$B(x_1,\ldots,x_{mn})(q_1,\ldots,q_m)=
\check{q}_1(x_1,\ldots,x_n)\cdots\check{q}_m(x_{(m-1)n+1},\ldots,x_{mn}),$$
is a continuous $m$-linear operator, so the map
$A \colon E^{mn} \longrightarrow\mathcal{P}(^m\mathcal{P}(^n E))$ given by $$A(x_1,\ldots,x_{mn})(q)=
\check{q}(x_1,\ldots,x_n)\check{q}(x_{n+1},\ldots,x_{2n})\cdots\check{q}(x_{(m-1)n+1},\ldots,x_{mn}),$$
is a continuous $mn$-linear operator, and $J_E^{m,n}(x)(q) = A(x^{mn})(q)$ for all $x \in E$ and $q \in {\cal P}(^nE)$. For $x\in E$, from (\ref{formula}) we get
$$\|J^{m,n}_E(x)\|=  \sup_{\substack{\|q\|\leq1 }}| J^{m,n}_E(x)(q)|  =\sup_{\substack{\|q\|\leq1 }}| q(x)^m|=\|x\|^{mn}.$$
\end{proof}

The classical linear commutative diagram is a very particular case of the next one.

\begin{proposition} For any $m, n, k, r, s\in\mathbb{N}$ and $P\in \mathcal{P}(^m E;F)$, the following diagram is commutative:
\begin{equation*}
\begin{gathered}
\xymatrix@C4pt@R20pt{
           E\ar@/_/[dd]_*{J^{r,mnk}_E} \ar[rrrrrrrrrr]^*{ P} & & & & &  & & & & & F \ar@/^/[dd]^*{J^{nrs,k}_{F}}&\\
&  &  &  &  & & & & &  & & & & &  &\\
         \mathcal{P}\left( ^r \mathcal{P}(^{mnk}E)\right)\ar[rrrrrrrrrr]^*{\Delta^{s}_{r}\left(\Delta^{n}_{k}P\right)}  & & & & & & & & & & \mathcal{P}(^{nrs} \mathcal{P}(^k F))
}
\end{gathered}
\end{equation*}
\end{proposition}
\begin{proof}
For $x\in E$ and $q\in\mathcal{P}(^k F)$,
\begin{align*}
\left(\Delta^{s}_{r}\left(\Delta^{n}_{k}P\right)\circ J^{r,mnk}_E\right)(x)(q)&=\Delta^{s}_{r}\left(\Delta^{n}_{k}P\right)\left( J^{r,mnk}_E(x)\right)(q)
=\left[ J^{r,mnk}_E(x)\left( \Delta^n_k P(q)\right)\right]^s\\
&=\left[\Delta^n_k P(q)(x)\right]^{rs}=q(p(x))^{nrs}= J^{nrs,k}_E(P(x))(q)\\
&=\left(J^{nrs,k}_F\circ P\right)(x)(q).
\end{align*}
\end{proof}

Adjoints of linear operators are always weak*-weak*-continuous. To generalize this fact we must say what we mean by the weak* topology on ${\cal P}(^kE)$.

\begin{definition}\label{ndef}\rm By the {\it weak* topology} on ${\cal P}(^kE)$ we mean the topology on ${\cal P}(^kE)$ induced by the weak* topology of  $\left(\widehat{\otimes}_{\pi}^{k,s} E \right)^{*}$ via the topological isomorphism
$$L_k^E\colon {\cal P}(^kE)\longrightarrow \left(\widehat{\otimes}_{\pi}^{k,s} E \right)^{*};~ L_k^E(q)=q_L,$$
where $\widehat{\otimes}_{\pi}^{k,s}E $ is the (completed) $k$-fold projective symmetric tensor product of $E$ and $q_L$ is the linearization of the polynomial $q$ (see \cite{floretnote, ryan}). As usual, for $x \in E$ we write $\otimes^k x = x \otimes \stackrel{(k)}{\cdots}\otimes x$.
\end{definition}

\begin{proposition}\label{w.p}
For all $n,k \in \mathbb{N}$ and $P \in {\cal P}(^mE;F)$, the polynomial $\Delta_k^n P\colon {\cal P}(^kF) \longrightarrow {\cal P}(^{mnk}E)$ is weak*-weak*-continuous.
\end{proposition}

\begin{proof}

Let $(q_{\lambda})_{\lambda}$ be a net in $\mathcal{P}(^kF)$ such that $q_{\lambda}\stackrel{w^*}{\longrightarrow} q\in\mathcal{P}(^k F)$. %Provaremos que $\Delta^n_k P(q_{\lambda})\stackrel{w^*}{\longrightarrow}\Delta^n_k P(q)$ em $\mathcal{P}(^{mnk}E)$.
By Definition \ref{ndef},
%$q_{\lambda}\stackrel{w^*}{\longrightarrow} q$ in $\mathcal{P}(^kF)$,
$(q_{\lambda})_L\stackrel{w^*}{\longrightarrow} q_L$ in $\left(\widehat{\otimes}^{k,s}_{\pi}F\right)^{*}$, that is,
\begin{equation*}
(q_{\lambda})_L(z)\longrightarrow q_L(z),~{\rm for ~every~} z\in\widehat{\otimes}^{k,s}_{\pi}F.
\end{equation*}
It follows that, if $w=\displaystyle\sum_{j=1}^{r}\lambda_j\otimes^{mnk}x_j\in\otimes^{mnk,s}_{\pi}E$, then
$$[q_{\lambda}(P(x_j))]^n=[(q_{\lambda})_L(\otimes^k P(x_j))]^n\longrightarrow[q_L(\otimes^k P(x_j))]^n=[q(P(x_j))]^n,$$
for $j=1,\ldots,r$, hence
\begin{equation}\label{2*ast}
\Delta^n_k P(q_{\lambda})(x_j)\longrightarrow \Delta^n_k P(q)(x_j),~{\rm para ~} j=1,\ldots,r.
\end{equation}
So,
\begin{align*}
\left[\Delta^n_k P(q_{\lambda})\right]_L(w)&=\left[\Delta^n_k P(q_{\lambda})\right]_L\left(\displaystyle\sum_{j=1}^{r}\lambda_j\otimes^{mnk}x_j\right)=
\displaystyle\sum_{j=1}^{r}\lambda_j\left[\Delta^n_k P(q_{\lambda})\right]_L(\otimes^{mnk}x_j)\\
&=\displaystyle\sum_{j=1}^{r}\lambda_j\Delta^n_k P(q_{\lambda})(x_j)\stackrel{(\ref{2*ast})}{\longrightarrow}\displaystyle\sum_{j=1}^{r}\lambda_j\Delta^n_k P(q)(x_j)=\displaystyle\sum_{j=1}^{r}\lambda_j\left[\Delta^n_k P(q)\right]_L(\otimes^{mnk}x_j)\\
&=\left[\Delta^n_k P(q)\right]_L\left( \displaystyle\sum_{j=1}^{r}\lambda_j\otimes^{mnk}x_j\right)=\left[\Delta^n_k P(q)\right]_L(w).
\end{align*}
%Therefore $\left[\Delta^n_k P(q_{\lambda})\right]_L(w)\longrightarrow\left[\Delta^n_k P(q)\right]_L(w)$, for all $w\in\otimes^{mnk,s}_{\pi}E$.
Since the net $(q_{\lambda})_L$ is bounded and
\begin{align*}
\|\left[\Delta^n_k P(q_{\lambda})\right]_L\| =\|\left[\Delta^n_k P(q_{\lambda})\right]^{\vee}\|\leq  \frac{(mnk)^{mnk}}{(mnk)!}\|\Delta^n_k P(q_{\lambda})\|\leq \frac{(mnk)^{mnk}}{(mnk)!}\| \Delta^n_k P\|\cdot\|q_{\lambda}\|^n
\end{align*}
for every $\lambda$, we have that the net $\left(\left[\Delta^n_k P(q_{\lambda})\right]_L\right)_{\lambda}$ is bounded as well. Combining this boundedness with the convergence $\left[\Delta^n_k P(q_{\lambda})\right]_L(w)\longrightarrow\left[\Delta^n_k P(q)\right]_L(w)$ for every $w\in\otimes^{mnk,s}_{\pi}E$, a standard approximation argument gives $\left[\Delta^n_k P(q_{\lambda})\right]_L\stackrel{w^*}{\longrightarrow}\left[\Delta^n_k P(q)\right]_L$ in $\left(\widehat{\otimes}^{mnk,s}_{\pi}E\right)^{*}$, that is, $\Delta^n_k P(q_{\lambda})\stackrel{w^*}{\longrightarrow}\Delta^n_k P(q)$ in $\mathcal{P}(^{mnk}E)$.\end{proof}

 It is well known that the converse of the result above holds in the linear case, that is, every weak*-weak* continuous linear operator from $F^*$ to $E^*$ is the adjoint of some operator from $E$ to $F$. This is also true for the Aron-Schottenloher adjoint of a homogeneous polynomial: if $T$ is a weak*-weak*-continuous linear operator from $F^*$ to ${\cal P}(^mE)$, then there exists a polynomial $P \in {\cal P}(^mE;F)$ such that $\Delta_1^1 P = P^* = T$ (see the proof of \cite[Corollary 2.3]{jmaaleticia}). Our next purpose is to show that this converse is no longer true in the generalized case.

\begin{lemma}\label{diagrama.delta.c}
If $P\in\mathcal{P}(^m E;F)$, then $L^E_{mk}\circ\Delta^1_k P\circ (L^F_k)^{-1}=\left[(\delta^k_F\circ P)_L\right]^{*}$, that is, the following diagram is commutative:

\begin{equation*}
\begin{gathered}
\xymatrix@C5pt@R15pt{
           \mathcal{P}(^k F)\ar@/_/[dd]_*{L^F_k} \ar[rrrrrrrrrr]^*{\Delta^1_k P} & & & & &  & & & & & {\cal P}(^{mk}E) \ar@/^/[dd]^*{L^E_{mk}}&\\
&  &  &  &  & & & & &  & & & & &  &\\
         \left(\widehat{\otimes}^{k,s}_{\pi} F\right)^{*}\ar[rrrrrrrrrr]^*{\left[(\delta^k_F\circ P)_L\right]^{*}}  & & & & & & & & & & \left(\widehat{\otimes}^{mk,s}_{\pi} E\right)^{*}
}
\end{gathered}
\end{equation*}
\end{lemma}
\begin{proof} Note that, since $\delta^k_F\circ P\in \mathcal{P}(^{mk}E;\widehat{\otimes}^{k,s}_{\pi}F)$, we have $(\delta^k_F\circ P)_L\in \mathcal{L}(\widehat{\otimes}^{mk,s}_{\pi}E;\widehat{\otimes}^{k,s}_{\pi}F)$. So, for every $\varphi\in(\widehat{\otimes}^{k,s}_{\pi}F)^{*}$,
\begin{align*}
\left[L^E_{mk}\circ\Delta^1_k P\circ (L^F_k)^{-1}\right](\varphi)&=L^E_{mk}(\Delta^1_k P((L^F_k)^{-1}(\varphi)))=L^E_{mk}(\Delta^1_k P(\varphi\circ\delta^k_F))\\
&=L^E_{mk}(\varphi\circ\delta^k_F\circ P)=(\varphi\circ\delta^k_F\circ P)_L\\
&=\varphi\circ(\delta^k_F\circ P)_L=\left[(\delta^k_F\circ P)_L\right]^{*}(\varphi).
\end{align*}
\end{proof}

Now we are in the position to show that, even in the case $n=1$, the converse of Proposition \ref{w.p} does not hold.  This establishes that, not only regarding proofs, but also regarding results, the generalized and the classical theories are not identical.

\begin{proposition}\label{w.rcp}
 Let $k>1$ and suppose that there exists a surjective polynomial $R\in\mathcal{P}(^{mk}E;\ell_1)$. Then there exists a weak*-weak*-continuous operator $T\in\mathcal{L}(\mathcal{P}(^k \ell_1);\mathcal{P}(^{mk}E))$ such that $T\not=\Delta^1_k P$ for every $P\in\mathcal{P}(^m E;\ell_1)$.
\end{proposition}
\begin{proof}
As $\widehat{\otimes}^{k,s}_{\pi}\ell_1$ is topologically isomorphic to $\widehat{\otimes}^{k}_{\pi}\ell_1$ \cite[Lemma 5.2]{arias} and the latter space is topologically isomorphic to $\ell_1$, we can consider a topological isomorphism $I \colon \ell_1 \longrightarrow \widehat{\otimes}^{k,s}_{\pi}\ell_1$. Since  $R\in\mathcal{P}(^{mk}E;\ell_1)$ is a surjective polynomial, $Q:=I\circ R$ is surjective as well. Define $$T:=(L^E_{mk})^{-1}\circ (Q_L)^{*}\circ L^F_k\in\mathcal{L}(\mathcal{P}(^kF);\mathcal{P}(^{mk}E)),$$
and note that $T$ is weak*-weak*-continuous. Suppose that there exists a polynomial $P\in\mathcal{P}(^m E;\ell_1)$ such that $T=\Delta^1_k P$. In this case, Lema \ref{diagrama.delta.c} gives $(Q_L)^{*}=\left[(\delta^k_F\circ P)_L\right]^{*}$, hence $ Q=\delta^k_F\circ P$, that is, $Q(x)=\otimes^k P(x)$ for every $x\in E$. Therefore,
$$ \widehat{\otimes}^{k,s}_{\pi}\ell_1=Q(E)\subseteq\otimes^{k,s}_{\pi}\ell_1\subseteq\widehat{\otimes}^{k,s}_{\pi}\ell_1,$$
from which we would conclude that the incomplete space $\otimes^{k,s}_{\pi}\ell_1$ is complete. This contradiction completes the proof.
\end{proof}

The next example completes the failure of the converse of Proposition \ref{w.p}.

\begin{example}\rm \label{exdif} Let $m,k \in \mathbb{N}$ be such that $mk$ is odd and $k > 1$. The map $$R\colon\ell_{mk}\longrightarrow\ell_1~,~R((\lambda_j)_{j\in\mathbb{N}})=(\lambda_j^{mk})_{j\in\mathbb{N}},$$
is a surjective continuous $mk$-homogeneous polynomial in both the real and complex cases. By Proposition \ref{w.rcp} there exists a weak*-weak*-continuous operator $T\in\mathcal{L}(\mathcal{P}(^k \ell_1);\mathcal{P}(^{mk}\ell_{mk}))$ such that $T\not=\Delta^1_k P$ for every $P\in\mathcal{P}(^m \ell_{mk};\ell_1)$.
%
%onde $\check{R}(X^1,\ldots,X^n)=(X_1\cdots X_n)=(X^1_j\cdots X^n_j)_{j\in\mathbb{N}}$, que é continua pela forma generalizada da desigualdade de Holder. Agora para mostrar que $R$ é sobrejetor, concideremos $(\alpha_j)_{j\in\mathbb{N}}$ uma sequencia em $\ell_1$ e notemos que para todo $j\in\mathbb{N}$, existe $\lambda_j\in \mathbb{N}$ tal que $\lambda_j^n=\alpha_j$ e $\sum_{j=1}^{\infty}|\lambda_j|^n=\sum_{j=1}^{\infty}|\lambda_j^n|=\sum_{j=1}^{\infty}|\alpha_j|<+\infty$.
\end{example}

%\begin{proposition}\label{equi.delta}
%Let $Q\in \mathcal{P}(^{mk}E;\widehat{\otimes}^{k,s}_{\pi}F)$, %$P\in\mathcal{P}(^m E;F)$.
%$$T_Q=\Delta^1_k P{\it~se~e~somente~se~}Q=\delta^k_F\circ P.$$
%\end{proposition}
%\begin{proof}
%Suponhamos que $T_Q=\Delta^1_k P$, então $T_Q(q)=\Delta^1_k P(q)$ para todo $q\in\mathcal{P}(^k F)$. Desde que $(L^E_{mk})^{-1}(q_L\circ Q_L)=q\circ P$, para todo $q\in\mathcal{P}(^k F)$, entao $q_L\circ Q_L=L^E_{mk}(q\circ P)\stackrel{{\rm lema \ref{delta}}}{=}q_L\circ (\delta ^k_F\circ P)_L$, para todo $q\in\mathcal{P}(^k F)$. Usando teorema de Hahn Banach $Q_L=(\delta ^k_F\circ P)_L$ e assim $Q=\delta ^k_F\circ P$. Reciprocamente suponhamos que $Q=\delta ^k_F\circ P$, então $T_Q(q)=(L^E_{mk})^{-1}(q_L\circ Q_L)\stackrel{{\rm lema \ref{delta}}}{=}q\circ P=\Delta^1_k P(q)$, para todo $q\in\mathcal{P}(^k F)$.
%\end{proof}

%\begin{example}
%Se $E=\ell_1$ e $n$ é um numero impar, sempre é possivel definir um %polinomio homogeneo continuo sobrejetor da seguinte maneira
%$$R\colon\ell_n\longrightarrow\ell_1;~R((\lambda_j)_{j\in\mathbb{N}})=(\%lambda_j^n)_{j\in\mathbb{N}}$$
%onde $\check{R}(X^1,\ldots,X^n)=(X_1\cdots X_n)=(X^1_j\cdots X^n_j)_{j\in\mathbb{N}}$, que é continua pela forma generalizada da desigualdade de Holder. Agora para mostrar que $R$ é sobrejetor, concideremos $(\alpha_j)_{j\in\mathbb{N}}$ uma sequencia em $\ell_1$ e notemos que para todo $j\in\mathbb{N}$, existe $\lambda_j\in \mathbb{N}$ tal que $\lambda_j^n=\alpha_j$ e $\sum_{j=1}^{\infty}|\lambda_j|^n=\sum_{j=1}^{\infty}|\lambda_j^n|=\sum_{j=1}^{\infty}|\alpha_j|<+\infty$.

%\end{example}

Now that we know, as expected, that linear results may fail in the general theory, we proceed in the opposite direction, namely, our next aim is to show that results that are unsuspected in the linear case hold in the generalized theory.

%\bigskip

Given $x^* \in E^*$ and $y \in F$, by $(x^*)^m\otimes b$ we mean the $m$-homogeneous polynomial defined by
$$\left((x^*)^m\otimes b\right)(x) = x^*(x)^mb {\rm ~for~every~} x \in E. $$
Linear combinations of polynomials of this kind are called {\it $m$-homogeneous polynomials of finite type} (see \cite{livrodineen}). According to \cite[page 42]{livrodineen}, a polynomial $P \in {\cal P}(^mE;F)$ is of finite type if and only if $P$ is a linear combination of polynomials of the type
$$x \in E \mapsto x_1^*(x) \cdots x_m^*(x)b,$$
 where $x_1^*, \ldots, x_m^* \in E^*$ and $b\in F$.

 A homogeneous polynomial, like any other nonlinear map between linear spaces, has {\it finite rank} if the subspace generated by its range in the target space is finite dimensional (see \cite{mujicatams}). It is well known that a polynomial $P \in {\cal P}(^mE;F)$ has finite rank if and only if $P$ is a linear combination of polynomials of the type $x \in E \mapsto q(x)b,$
 where $q \in {\cal P}(^mE)$ and $b\in F$. Of course, polynomials of finite type have finite rank.
%
%
%
%
%
%\begin{lemma}
%Given $E$, $F$ Banach spaces and $n\in\mathbb{N}$,
%$$\mathcal{P}_{f}(^n E;F)=\left\{ \displaystyle\sum_{j=1}^{k}\varphi_{1,j}\varphi_{2,j}\ldots\varphi_{n,j}\otimes b_j:\varphi_{i,j}\in E^{*},~ b_j\in F,~k\in\mathbb{N}\right\}.$$
%\end{lemma}
%\begin{proof}
%é consequencia de um resultado de \cite[pag 42] {livrodineen}
%\end{proof}

According to what happens with the adjoint of a linear operator and with the Aron-Schottenloher adjoint of a homogeneous polynomial (see \cite[Lemma 2.1]{jmaaleticia}), it is expected that if the polynomial $P$ is of finite type (has finite rank, respectively), then $\Delta_k^n P$ is of finite type (has finite rank, respectively) as well. For linear operators, being of finite type is the same of being of finite rank, and this is the reason why the following more general property has been disclosed only in our generalized setting.

\begin{proposition} \label{delta.ida}
If the polynomial $P \in {\cal P}(^mE;F)$ has finite rank, then $\Delta_k^n P$ is of finite type for all $k,n \in \mathbb{N}$.
\end{proposition}
\begin{proof} Let $l \in \mathbb{N}$, $P_1,\ldots, P_l\in\mathcal{P}(^m E)$ and $b_1,\ldots, b_l\in F$ be such that $P(x) =\sum\limits_{j=1}^{l}P_j(x)b_j$ for every $x\in E$. For $q\in\mathcal{P}(^k F)$ and $x\in E$, from the Leibniz Formula \cite[Theorem 1.8]{mujicalivro} and the Multinomial Formula \cite[page 33]{berge}, we have
\begin{align*}
[\Delta^n_k P](q)(x)&=q(P(x))^n =\left[q\left(\sum_{j=1}^{l}P_j(x) b_j\right)\right]^n=\left[\check{q} \left(\sum_{j=1}^{l}P_j(x) b_j\right)^{(k)} \right]^n\\
&=\left[\sum_{\atop
{k_1 + \cdots +k_l = k} }
\frac{k!}{\prod\limits_{i=1}^{l} k_i!} {P_1(x)^{k_1}\cdots P_l(x)^{k_l}\check{q}\left(b_1^{(k_1)},\ldots, b_l^{(k_l)}\right)}\right]^n\\
&=\displaystyle\sum_{\atop
{\sum\limits_{\atop
{k_1 + \cdots +k_l = k} }^{}} \alpha_{k_1,\ldots, k_l} =n}
\left(\frac{n!}{\displaystyle\prod_{k_1+\cdots+k_l=k} \alpha_{k_1,\ldots,k_l}!}\right) {\prod_{k_1+\cdots+k_l=k}^{} C_{k_1,\ldots,k_l}},
\end{align*}
where $C_{k_1,\ldots,k_l}= \left[\frac{k!}{\prod\limits_{i=1}^{l} k_i!} \left(\prod\limits_{i=1}^{l} P_i(x)^{k_i} \right)\check{q}\left(b_1^{(k_1)},\ldots, b_l^{(k_l)}\right)\right]^{\alpha_{k_1,\ldots,k_l}}$. Then
\begin{align*}
\displaystyle\prod_{k_1+\cdots+k_l=k}^{}& C_{k_1,\ldots,k_l}= \displaystyle\prod_{k_1+\cdots+k_l=k}^{}\left[\frac{k!}{\prod\limits_{i=1}^{l} k_i!} \left(\prod_{i=1}^{l} P_i(x)^{k_i} \right)\check{q}\left(b_1^{(k_1)},\ldots, b_l^{(k_l)}\right)\right]^{\alpha_{k_1,\ldots,k_l}}\\
&= \displaystyle\prod_{k_1+\cdots+k_l=k}^{}\left[\left(\frac{k!}{\prod\limits_{i=1}^{l} k_i!}\right)^{\alpha_{k_1,\ldots,k_l}} \left(\prod_{i=1}^{l} P_i(x)^{k_i \alpha_{k_1,\ldots,k_l}} \right)\check{q}\left(b_1^{(k_1)},\ldots, b_l^{(k_l)}\right)^{\alpha_{k_1,\ldots,k_l}}\right]\\
&=(k!)^n K_{(\alpha_{k_1,\ldots,k_l}:k_1+\cdots+k_l=k)} P_{(\alpha_{k_1,\ldots,k_l}:k_1+\cdots+k_l=k)}
 (x)Q_{(\alpha_{k_1,\ldots,k_l}:k_1+\cdots+k_l=k)}(q),
\end{align*}
where $$K_{(\alpha_{k_1,\ldots,k_l}:k_1+\cdots+k_l=k)}= \displaystyle\prod_{k_1+\cdots+k_l=k}^{} \left(\frac{1}{\left(\prod\limits_{i=1}^{l} k_i!\right)^{\alpha_{k_1,\ldots,k_l}}}\right),$$
$$P_{(\alpha_{k_1,\ldots,k_l}:k_1+\cdots+k_l=k)}(x)=\displaystyle\prod_{k_1+\cdots+k_l=k}^{} \left(  \prod_{i=1}^{l} P_i(x)^{k_i \alpha_{k_1,\ldots,k_l}} \right),$$
$$Q_{(\alpha_{k_1,\ldots,k_l}:k_1+\cdots+k_l=k)}(q)=\displaystyle\prod_{k_1+\cdots+k_l=k}^{}\left[ \check{q}\left(b_1^{(k_1)},\ldots, b_l^{(k_l)}\right)^{\alpha_{k_1,\ldots,k_l}} \right].$$
Therefore,
\begin{align*}
&[\Delta^n_k P](q)(x)= \\
& \displaystyle\sum_{\atop
{\sum\limits_{\atop
{k_1 + \cdots +k_l = k} }^{}} \alpha_{k_1,\ldots, k_l} =n}^{}
 { \Theta_{(\alpha_{k_1,\ldots,k_l}:k_1+\cdots+k_l=k)} P_{(\alpha_{k_1,\ldots,k_l}:k_1+\cdots+k_l=k)}(x)
 Q_{(\alpha_{k_1,\ldots,k_l}:k_1+\cdots+k_l=k)}(q)}=\\
 & \left(\displaystyle\sum_{\atop
{\sum\limits_{\atop
{k_1 + \cdots +k_l = k} }^{}} \alpha_{k_1,\ldots, k_l} =n}^{}
 { \Theta_{(\alpha_{k_1,\ldots,k_l}:k_1+\cdots+k_l=k)} P_{(\alpha_{k_1,\ldots,k_l}:k_1+\cdots+k_l=k)}
 Q_{(\alpha_{k_1,\ldots,k_l}:k_1+\cdots+k_l=k)}}\right)(q)(x),
\end{align*}
where $$\Theta_{(\alpha_{k_1,\ldots,k_l}:k_1+\cdots+k_l=k)}=\left(\frac{n!(k!)^n}{\displaystyle\prod_{k_1+\cdots+k_l=k}^{} \alpha_{k_1,\ldots,k_l}!}\right) K_{(\alpha_{k_1,\ldots,k_l}:k_1+\cdots+k_l=k)}.$$
%from which it follows that
%\begin{align*}
%&\Delta^n_k P = \\
%& \displaystyle\sum_{\atop
%{\sum_{\atop
%{k_1 + \cdots +k_l = k} }^{}} \alpha_{k_1,\ldots, k_l} =n}^{}
% { \Theta_{(\alpha_{k_1,\ldots,k_l}:k_1+\cdots+k_l=k)}
% Q_{(\alpha_{k_1,\ldots,k_l}:k_1+\cdots+k_l=k)}}\otimes P_{(\alpha_{k_1,\ldots,k_l}:k_1+\cdots+k_l=k)}
%\end{align*}

Note that $P_{(\alpha_{k_1,\ldots,k_l}:k_1+\cdots+k_l=k)}\in\mathcal{P}(^{mnk}E)$ and   the polynomial $Q_{(\alpha_{k_1,\ldots,k_l}:k_1+\cdots+k_l=k)}\in\mathcal{P}(^n \mathcal{P}(^k F))$ is of finite type because $$Q_{(\alpha_{k_1,\ldots,k_l}:k_1+\cdots+k_l=k)}=\displaystyle\prod_{k_1+\cdots+k_l=k}^{}\left[\Psi_{\left(b_1^{(k_1)},\ldots, b_l^{(k_l)}\right)}\right]^{\alpha_{k_1,\ldots,k_l}},$$ where
$$
\Psi_{\left(b_1^{(k_1)},\ldots, b_l^{(k_l)}\right)} \colon \mathcal{P}(^k F) \longrightarrow \mathbb{K},~
\Psi_{\left(b_1^{(k_1)},\ldots, b_l^{(k_l)}\right)}(q)=\check{q}\left(b_1^{(k_1)},\ldots, b_l^{(k_l)}\right)$$
is a continuous linear functional. It follows that $\Delta_k^n P$ is of finite type.
\end{proof}

We finish this section with a partial converse of the proposition above. The proof is illustrative of the interplay between linear and nonlinear arguments.

\begin{proposition}\label{delta.r} The following are equivalent for a polynomial  $P\in\mathcal{P}(^m E;F)$:\\
{\rm (a)} $P$ has finite rank.\\
{\rm (b)} $\Delta^1_k P$ is a finite rank operator for every $k \in \mathbb{N}$.\\
{\rm (c)} $\Delta^1_k P$ is a finite rank operator for some $k \in \mathbb{N}$.
\end{proposition}
\begin{proof} (a) $\Longrightarrow$ (b) follows from Proposition \ref{delta.ida} and  (b) $\Longrightarrow$ (c) is obvious. To prove (c) $\Longrightarrow$ (a), assume that $\Delta^1_k P\in\mathcal{L}(\mathcal{P}(^k F);\mathcal{P}(^{mk} E))$ has finite rank. Since the class of finite rank operators is an operator ideal,   calling on Lemma \ref{diagrama.delta.c} once again we conclude that the linear operator $\left[(\delta^k_F\circ P)_L\right]^{*}\in \mathcal{L}\left(\left(\widehat{\otimes}^{k,s}_{\pi}F\right)^{*};\left(\widehat{\otimes}^{mk,s}_{\pi}E\right)^{*}\right)$ has finite rank. But the ideal of finite rank operators is completely symmetric \cite[Proposition 4.4.7]{pietschlivro}, so $(\delta^k_F\circ P)_L\in \mathcal{L}\left(\widehat{\otimes}^{mk,s}_{\pi}E;\widehat{\otimes}^{k,s}_{\pi}F\right)$ has finite rank.  Now, combining \cite[Proposition 3.2.b]{Geraldo.D.P} and \cite[Proposition 3.1.b]{mujicatams} we get that the polynomial $\delta^k_F\circ P$ has finite rank, from which we conclude that the range of $\delta^k_F\circ P$ does not contain infinitely many linearly independent vectors.

Suppose that $ P$ has not finite rank, that is, there are $(x_j)_{j=1}^\infty \subseteq  E$ such that the set $\left\{P(x_1), P(x_2), P(x_3),\ldots\right\}$ is linearly independent in $F$. By \cite[Proposition 1.1]{ryan.livro} it follows that the set $\left\{P(x_{i_1})\otimes P(x_{i_2})\otimes\cdots\otimes P(x_{i_k}):~i_1,i_2,\ldots i_k\in\mathbb{N}\right\}$ is linearly independent in $\otimes^k F$, so its subset $\left\{ \otimes^k P(x_1),\otimes^k P(x_2),\otimes^k P(x_3), \ldots \right\}$ is linearly independent in the range of $\delta^k_F\circ P $. This contradiction shows that $P$ has finite rank.
\end{proof}

Many other results related to the ones proved in this section can be obtained. We refrain from going further because we believe that thus far the reader is convinced that the $\Delta_k^n P$'s are genuine generalizations of $u^*$ and $P^*$ and that the general theory deserves to be investigated.

\section{Applications to composition operators}
In this section we show how our generalized adjoints can be used to take several results on composition operators due to Lindstr\"om and Schl\"uchtermann \cite{lindstrom} beyond their original scope.

Operator ideals will be taken in the sense of Pietsch \cite{df, pietschlivro}, ideals of homogeneous polynomials (polynomial ideals) in the sense of Garc\'ia and Floret \cite{domingoklaus}, two-sided polynomial ideals in the sense of \cite{ewertonjmaa} and polynomial hyper-ideals in the sense of \cite{arxiv}. For the sake of the reader, we recall these concepts next.

\begin{definition}\rm Let ${\cal Q}$ be a subclass of the class of homogeneous polynomials between Banach spaces such that, for every $m$ and any Banach spaces $E$ and $F$, the component
$${\cal Q}(^mE;F) := {\cal P}(^mE;F) \cap {\cal Q} $$
is a linear subspace of ${\cal P}(^mE;F)$ containing the polynomials of finite type. The class $\cal Q$ is said to be:\\
(a) A {\it polynomial ideal} if $t \circ P \circ u \in {\cal Q}(^mE;H)$ whenever $t \in {\cal L}(G;H)$, $P \in {\cal Q}(^mF;G)$ and $u \in {\cal L}(E;F)$.\\
(b) A {\it polynomial hyper-ideal} if $t \circ P \circ Q \in {\cal Q}(^{mn}E;H)$ whenever $t \in {\cal L}(G;H)$, $P \in {\cal Q}(^mF;G)$ and $Q \in {\cal P}(^nE;F)$.\\
(c) A {\it polynomial two-sided ideal} if $R \circ P \circ Q \in {\cal Q}(^{mnk}E;H)$ whenever $R \in {\cal P}(^kG;H)$, $P \in {\cal P}(^mF;G)$ and $Q \in {\cal P}(^nE;F)$.

Suppose that there is a function $\|\cdot\|_{\cal Q} \colon {\cal Q} \longrightarrow \mathbb{R}$ whose restriction to each component ${\cal Q}(^mE;F)$ is a complete norm and such that $\|\lambda \in \mathbb{K} \mapsto \lambda^m\|_{\cal Q} = 1$ for every $m$. $({\cal Q}, \|\cdot\|_{\cal Q})$ is said to be:\\
(a') A {\it Banach polynomial ideal} if, in (a), $\|t \circ P \circ u\|_{\cal Q} \leq \|t\|\cdot\|P\|_{\cal Q}\cdot \|u\|$.\\
(b') A {\it Banach polynomial hyper-ideal ideal} if, in (b), $\|t \circ P \circ Q\|_{\cal Q} \leq \|t\|\cdot\|P\|_{\cal Q}\cdot \|Q\|^m$.\\
(c') A {\it Banach polynomial two-sided ideal} if, in (c), $\|R \circ P \circ Q\|_{\cal Q} \leq \|R\|\cdot\|P\|_{\cal Q}^k\cdot \|Q\|^{mk}$.
\end{definition}

Just to illustrate that the three concepts above are worth to be considered, we mention that: (i) the class of nuclear polynomials is a polynomial ideal that fails to be a hyper-ideal; (ii) the class of weakly compact polynomials is a hyper-ideal that fails to be a two-sided ideal; (iii) the class of compact polynomials is a two-sided ideal.

The proof of the following lemma is easy and we omit it.

\begin{lemma}\label{lematc1} Let $E, F, E_1, F_1$ be Banach spaces and $m, r, s\in \mathbb{N}$. If $R\in \mathcal{P}(^r E;F)$, $B\in \mathcal{P}(^s E_1;F_1)$ are non-zero polynomials and $\left(\mathcal{Q},\|\cdot\|_{\mathcal{Q}}\right)$ is a two-sided polynomial Banach ideal, then the map
$$S_{RB}\colon \mathcal{Q}(^m F_1;E)\longrightarrow \mathcal{Q}(^{mrs} E_1;F)~,~S_{RB}(P)=R\circ P\circ B,$$
is a well defined continuous $r$-homogeneous polynomial.
\end{lemma}

The dual of an operator ideal $\cal I$ is the operator ideal defined by
$${\cal I}^{\rm dual}(E;F) = \{u \in {\cal L}(E;F) : u^* \in {\cal I}(F^*;E^*)\},$$
and the polynomial dual of an operator ideal $\cal I$ is the polynomial ideal defined by (see \cite{monat})
$${\cal I}^{\mathcal{P}-{\rm dual}}(^mE;F) = \{P \in {\cal P}(^mE;F) : P^* \in {\cal I}(F^*;{\cal P}(^mE)\}.$$
Similarly, given a polynomial ideal $\cal Q$ and $k,n \in \mathbb{N}$, we define
$$\Delta_k^n {\cal Q}(^mE;F) = \{P \in {\cal P}(^mE;F) : \Delta_k^nP \in {\cal Q}(^n{\cal P}(^kF); {\cal P}(^{mnk}E))\}.$$
To recover the original concepts, note that the linear component ${\cal Q}_1$ of a polynomial ideal $\cal Q$ is an operator ideal and %So, $\Delta_1^1 {\cal Q} = {\cal Q}_1^{\rm dual}$, or, more precisely,
$$\Delta_1^1 {\cal Q}(E;F) = \Delta_1^1 {\cal Q}(^1E;F) = {\cal Q}_1^{\rm dual}(E;F),~ \Delta_1^1 {\cal Q}(^m E;F)={\cal Q}_1^{\mathcal{P}-{\rm dual}}(^m E;F).$$

%$$\Delta_1^1 {\cal Q}(^m E;F) = \Delta_1^1 {\cal Q}(^1E;F) = {\cal Q}_1^{\rm dual}(E;F). $$

\begin{remark}\rm The algebraic structure of the class of polynomials $\Delta_n^k {\cal Q}$ shall be investigated in a forthcoming paper. For the moment we just mention that, from Proposition \ref{delta.ida} and Proposition \ref{delta.r}, we have the following: (i) for every polynomial ideal $\cal Q$,  $\Delta^n_k \mathcal{Q}$ contains the polynomials of finite rank for all $k$ and $n$; (ii) if ${\cal P}_{\cal F}$ stands for the ideal of finite rank polynomials, then  $\mathcal{P}_{\mathcal{F}}=\Delta^1_k{\cal P}_\mathcal{F}$ for every $k$.
\end{remark}

The next result is a polynomial version of \cite[Proposition 2.1]{lindstrom}.

\begin{theorem} \label{i.thm} Let $E, F, E_1, F_1$ be Banach spaces, $m, r, s\in \mathbb{N}$, $R\in \mathcal{P}(^r E;F)$ and $B\in \mathcal{P}(^s E_1;F_1)$ be non-zero polynomials. If $\mathcal{R}$ is a polynomial hyper-ideal and $\left(\mathcal{Q},\|\cdot\|_{\mathcal{Q}}\right)$ is a two-sided polynomial Banach ideal such that the polynomial
$$S_{RB}\colon\mathcal{Q}(^m F_1;E)\longrightarrow \mathcal{Q}(^{mrs} E_1;F)~,~S_{RB}(P)=R\circ P\circ B,$$
belongs to $\mathcal{R}$, then:\\
{\rm (a)} $R\in\mathcal{R}(^r E;F)$.\\
{\rm (b)} $B \in \Delta^{mr}_1 \mathcal{R}(^{s}E_1;F_1)$.
\end{theorem}

\begin{proof}
(a) Choose $\varphi\in F_1^*$ and $z\in E_1$ such that
$\varphi(B(z))=1$ and define
$$u_{\varphi} \colon E \longrightarrow \mathcal{Q}(^{m}F_1;E)~,~
u_{\varphi}(x)=\varphi^{m}\otimes x {\rm~(meaning ~that~} u_\varphi(x)(y) = \varphi(y)^mx),$$
$$t_{z} \colon \mathcal{Q}(^{mrs}E_1;F) \longrightarrow F~,~t_{z}(P)=P(z).$$
Using the properties of the ideal norm $\|\cdot\|_{\cal Q}$ we get $\|u_{\varphi}(x)\|_{\mathcal{Q}}=\|\varphi^{m}\otimes x\|_{\mathcal{Q}}=\|\varphi\|^{m}\|x\|$ and
$$\|t_{z}(P)\|=\|P(z)\|\leq \|P\|\cdot\|z\|^{mrs}\leq \|P\|_{\mathcal{Q}}\cdot\|z\|^{mrs},$$
from which we conclude that $u_{\varphi}\in \mathcal{L}(E;\mathcal{Q}(^m F_1;E))$ and $t_z\in\mathcal{L}(\mathcal{Q}(^{mrs}E_1;F);F)$. For every $x\in E$,
$$(t_z\circ S_{RB}\circ u_{\varphi})(x)=R([\varphi(B(z)]^{m}x))=R(x),$$
 proving that $R = t_z \circ S_{RB} \circ u_\varphi \in\mathcal{R}(^r E;F)$ by the ideal property of $\cal R$.

\medskip

\noindent(b) Choose $z\in E$ and $\psi\in F^{*}$ such that $\psi(R(z))=1$ and define
$$w_{z} \colon F_1^{*} \longrightarrow \mathcal{Q}(^{m}F_1;E)~,~w_{z}(\varphi)=\varphi^m\otimes z,$$
$$v_{\psi} \colon \mathcal{Q}(^{mrs}E_1;F) \longrightarrow \mathcal{P}(^{mrs}E_1), v_{\psi}(P)=\psi\circ P.$$
Similarly to the proof of (a) we have that $w_z$ is a continuous $m$-homogeneous polynomial and $v_\varphi$ is a continuous linear operator.
%, because $$\|w_{z}(\varphi)\|=\|\varphi^m\otimes z\|_{\mathcal{Q}}=\|\varphi\|^m\|z\|$$
%$$\|v_{f}(P)\|=\|f\circ P\| \leq \|f\|\|P\|\leq \|f\|\|P\|_{\mathcal{Q}}.$$
%
%
%
For every $\varphi\in F_1^{*}$,
$$
(v_\psi\circ S_{RB}\circ w_z)(\varphi) =\psi(R(z)) \Delta^{mr}_1 B(\varphi)=\Delta^{mr}_1 B(\varphi),
$$
and therefore $\Delta^{mr}_1 B\in\mathcal{R}(^{mr}F_1^{*};\mathcal{P}(^{mrs}E_1))$.
\end{proof}

Our next purpose is to give another polynomial version of \cite[Proposition 2.1]{lindstrom}, with a weaker assumption on the polynomial ideal $\cal Q$. A short preparation is needed.

Given $q \in {\cal P}(^mE)$ and $ y \in F$, by $q \otimes y$ we denote the rank 1 $m$-homogeneous polynomial given by
$$(q \otimes y)(x) = q(x)y {\rm ~for~every~} x \in E. $$
%It is well known that a polynomial is of finite rank if and only if it is a linear combination of polynomials of this type.

We say that Banach polynomial ideal $({\cal Q}, \|\cdot\|_{\cal Q})$ {\it contains the finite rank polynomials strongly} if for each $m \in \mathbb{N}$ there exists a constant $K_m$ such that for any Banach spaces $E$ and $F$, $q \in {\cal P}(^mE)$ and $ y \in F$, we have $P \otimes y \in {\cal Q}(^mE;F)$ and $\|q \otimes y\|_{\mathcal{Q}} \leq K_m \|q\|\cdot\|y\|$.
%\textcolor{blue}{(Aqui $K$, só não pode depender do polinomio $P$)}.

Banach polynomial hyper-ideals contain the finite rank polynomials strongly \cite{ewertonjmaa}.

Similarly to the definition of $\Delta^n_k\mathcal{Q}$, given an operator ideal $\mathcal{A}$ and $k\in\mathbb{N}$, define
$$\Delta_k^1 {\cal A}(^mE;F) = \{P \in {\cal P}(^mE;F) : \Delta_k^1 P \in {\cal A}({\cal P}(^kF); {\cal P}(^{mk}E))\}.$$

As well as Theorem \ref{i.thm}, the linear case of the next result recovers \cite[Proposition 2.1]{lindstrom}.
\begin{theorem} Let $E, E_1, F, F_1$ be Banach spaces, $R\in \mathcal{L}(E;F)$ and $B\in \mathcal{L}(E_1;F_1)$ be non-zero operators. If $\mathcal{A}$ is an operator ideal and $(\mathcal{Q},\|\cdot\|_{\mathcal{Q}})$ is a Banach polynomial ideal such that the bounded linear operator
$$ S_{RB}\colon\mathcal{Q}(^{m}F_1;E) \longrightarrow \mathcal{Q}(^{m}E_1;F)~,~
S_{RB}(P)=R\circ P\circ B, $$
%\[
%\begin{array}
%[c]{cccccc}%
% S_{RB}\colon\mathcal{Q}(^{m}F_1;E) \longrightarrow \mathcal{Q}(^{m}E_1;F) ;
%S_{RB}(P)=R\circ P\circ B, &&
%\end{array}
%\]
belongs to $\mathcal{A}$, then:\\
{\rm (a)} $R\in \mathcal{A}(E;F)$.\\
{\rm (b)} $B \in \Delta_m^1{\cal A}(E_1;F_1)$ if $(\mathcal{Q},\|\cdot\|_{\mathcal{Q}})$ contains the finite rank polynomials strongly.
\end{theorem}

\begin{proof} The proof of (a) is similar to the proof of Theorem \ref{i.thm}(a). To prove (b), take $z\in E$ and $\varphi\in F^{*}$ such that $\varphi(R(z))=1$. Since $(\mathcal{Q},\|\cdot\|_{\mathcal{Q}})$ is a Banach polynomial ideal containing the finite rank polynomials strongly, the maps
$$
w_{z} \colon \mathcal{P}(^{m}F_1) \longrightarrow \mathcal{Q}(^{m}F_1;E)~,~w_{z}(q)=q\otimes z, {\rm ~and}$$
$$
v_{\varphi} \colon \mathcal{Q}(^{m}E_1;F) \longrightarrow \mathcal{P}(^{m}E_1)~,~ v_{\varphi}(P)=\varphi\circ P,$$ are well defined linear operators. Their continuity follow from the inequalities
$$\|w_{z}(q)\|=\|q\otimes z\|_{\mathcal{Q}}\leq K_m\|q\|\cdot \|z\|~,~\|v_{\varphi}(P)\|=\|\varphi\circ P\| \leq \|\varphi\|\cdot\|P\|\leq \|\varphi\|\cdot\|P\|_{\mathcal{Q}}.$$
It is immediate that $
v_\varphi\circ S_{RB}\circ w_z =\Delta^{1}_m B$, hence $\Delta ^1_m B\in\mathcal{A}(\mathcal{P}(^{m}F_1);\mathcal{P}(^{m}E_1))$.
\end{proof}

To give one last application of our generalized adjoints, we extend the classes of operators introduced in \cite[Proposition 2.2]{lindstrom}.

\begin{definition}\rm  Let $\mathcal{A}$ be an operator ideal and $\left(\mathcal{Q},\|\cdot\|_{\mathcal{Q}}\right)$ be a Banach polynomial ideal containing the finite rank polynomials strongly. For $m \in \mathbb{N}$, we say that an operator $R \in {\cal L}(E;F)$ belongs to  $\mathcal{A}^{comp}_{m,left}$  if for any Banach spaces $E_1$ and $F_1$ and any operator $B \in \Delta_m^1\mathcal{A}(E_1;F_1)$, the operator
$$S_R \colon \mathcal{Q}(^m F_1;E) \longrightarrow \mathcal{Q}(^m E_1;F)~,~S_R(P) = R \circ P \circ B, $$
belongs to ${\cal A}$.
%\begin{align*}
%\mathcal{A}^{comp}_{m,left}(E;F):=&\{ R\in \mathcal{L}(E;F)| \textrm{para quaisquer espaços de Banach}\  E_1, F_1 \\
%& \textrm{e qualquer}\ B\in \mathcal{A}^{(m)}(E_1;F_1),
%\begin{array}{cccc}
%S_R \ : & \! \mathcal{Q}(^m F_1;E) & \! \longrightarrow
%& \! \mathcal{Q}(^m E_1;F) \\
%& \! p & \! \longmapsto
%& \! R\circ p\circ B
%\end{array}
%\in \mathcal{A}\}
%\end{align*}
\end{definition}
According to the terminology of \cite{lindstrom}, we have $\mathcal{A}^{comp}_{1,left}=\mathcal{A}^{comp}_{left}$, so the next result recovers \cite[Proposition 2.2]{lindstrom} as a particular case.

The definition of injective Banach polynomial ideal is the obvious one, namely: a polynomial ideal $\mathcal{Q}$ is {\it injective} if  $P \in {\cal Q}(^mE;F)$ whenever $P \in {\cal P}(^mE;F)$, $I \colon F \longrightarrow G$ is a metric injection and $I \circ P \in {\cal Q}(^mE;G)$. And a normed polynomial ideal $\left(\mathcal{Q},\|\cdot\|_{\mathcal{Q}}\right)$ is {\it injective}, if, in addition, $\|P\|_{\cal Q}=\|I \circ P\|_{\cal Q} $.

\begin{proposition} Let $\mathcal{A}$ be a closed operator ideal and $\left(\mathcal{Q},\|\cdot\|_{\mathcal{Q}}\right)$ be a Banach polynomial ideal containing the finite rank polynomials strongly. Then:\\
{\rm (a)} For every  $m\in \mathbb{N}$, $\mathcal{A}^{comp}_{m,left}$ is a closed operator ideal contained in $\mathcal{A}$.\\
{\rm (b)} If $\mathcal{A}$ and $\left(\mathcal{Q},\|\cdot\|_{\mathcal{Q}}\right)$ are injective ideals, then $\mathcal{A}^{comp}_{m,left}$ is injective as well for every $m\in \mathbb{N}$.
\end{proposition}

\begin{proof} We skip the proof that $\mathcal{A}^{comp}_{m,left}(E;F)$ is a linear subspace of $\mathcal{L}(E;F)$. Let us prove that $\mathcal{A}^{comp}_{m,left}(E;F)$ contains the operators of finite type. Let $\varphi \in E^*$, $ b\in F$, $E_1$, $F_1$ be Banach space and $B\in\Delta^1_m\mathcal{A}(E_1;F_1)$. The assumptions on $\left(\mathcal{Q},\|\cdot\|_{\mathcal{Q}}\right)$ guarantee that the maps
$$\delta_{\varphi} \colon \mathcal{Q}(^m F_1;E) \longrightarrow \mathcal{P}(^m F_1)~,~\delta_{\varphi}(P) = \varphi\circ P,$$
$$M_b \colon \mathcal{P}(^m E_1) \longrightarrow \mathcal{Q}(^m E_1;F),~ M_{b}(q) = q\otimes b,$$
are well defined continuous linear operators. For all $P \in \mathcal{Q}(^m F_1;E)$ and $x\in E_1$,
\begin{align*}S_{\varphi\otimes b}(P)(x)&=[(\varphi\otimes b)\circ P\circ B](x)=[(\varphi\circ P\circ B)\otimes b](x)= (\varphi \circ P \circ B)(x) b\\
&= (M_b\circ \Delta^1_m B\circ \delta_{\varphi})(P)(x),
\end{align*}
proving that $S_{\varphi\otimes b} = (M_b\circ \Delta^1_m B\circ \delta_{\varphi})$ belongs to
% \[
%S_{\varphi\otimes b}\colon\mathcal{Q}(^m F_1;E) \longrightarrow \mathcal{Q}(^m E_1;F) ~;~S_{\varphi\otimes b}(P)=(\varphi\otimes b)\circ P\circ B,\]
%is a well defined linear operator and
% $S_{\varphi\otimes b}(P)(x)=[(\varphi\circ P\circ B)\otimes b](x)$ . So we define the continuous linear operators
% where
%$(M_b\circ \Delta^1_m B\circ \delta_{\varphi})(P)=S_{\varphi\otimes b}(P)$. Then $S_{\varphi\otimes b}\in
$\mathcal{A}$, that is, $\varphi \otimes b$ belongs to $\mathcal{A}^{comp}_{m,left}$.

To prove the ideal property, $G$, $H$ be Banach spaces, $A\in \mathcal{L}(G;E)$, $R\in \mathcal{A}^{comp}_{m,left}(E;F)$ and $C\in \mathcal{L}(F;H)$. Also, let $E_1$, $F_1$ be Banach spaces and $B\in \Delta^1_m\mathcal{A}(E_1;F_1)$. As we have done before, the maps %and the continuous linear operator
%\[
%\begin{array}
%[c]{cccccc}%
% S_{C\circ R\circ A}\colon\mathcal{Q}(^m F_1;G) \longrightarrow \mathcal{Q}(^m E_1;H) ;
%S_{C\circ R\circ A}(P)=C\circ R\circ A\circ P\circ B. &&
%\end{array}
%\]
$$\delta_{A} \colon \mathcal{Q}(^m F_1;G) \longrightarrow \mathcal{Q}(^m F_1;E)~,~\delta_{A}(P) = A\circ P, $$
$$\delta_C \colon \mathcal{Q}(^m E_1;F) \longrightarrow \mathcal{Q}(^m E_1;H)~,~ \delta_C(Q) = C\circ Q,$$
are well defined continuous linear operators and
$$S_{C\circ R\circ A}(P)= C\circ R\circ A\circ P\circ B = (\delta_C\circ S_R\circ\delta_A)(P) $$
for every $P \in \mathcal{Q}(^m F_1;G)$,
which proves that $S_{C\circ R\circ A}$ belongs to $\mathcal{A}$.

Now we check that $\mathcal{A}^{comp}_{m,left}(E;F)$ is a closed subspace of ${\cal L}(E;F)$. To do so, let $R\in\mathcal{L}(E;F)$ and $(R_n)_n$ be a sequence in $\mathcal{A}^{comp}_{m,left}(E;F)$ such that $R_n\longrightarrow R$ in the usual sup norm. Then each $S_{R_n}$ belongs to ${\cal A}$ and, for every $n$, $$
\|S_{R_n}-S_{R}\|
=\sup_{\substack{\|P\|_{\mathcal{Q}}\leq1 }}\|(R_{n}-R)\circ P\circ B\|_{\mathcal{Q}}
\leq \sup_{\substack{\|P\|_{\mathcal{Q}}\leq1 }}\|R_n-R\|\cdot\| P\|_{\mathcal{Q}}\cdot\|B\|^m
=\|R_n-R\|\cdot\|B\|^{m},
$$
so $S_{Rn} \longrightarrow S_R \in {\cal A}({\cal Q}(^mF_1;E); {\cal Q}(^mE_1;F))$ because  the ideal $\mathcal{A}$ is closed. Therefore, $R\in\mathcal{A}^{comp}_{m,left}(E;F)$.

 In order to prove that $\mathcal{A}^{comp}_{m,left}$ is contained in $ \mathcal{A}$, let
 $R\in \mathcal{A}^{comp}_{m,left}(E;F)$ be given. Choose $E_1=F_1=\mathbb{K}$,  $B=id_{\mathbb{K}}$ and consider the continuous linear operators
$$\delta\colon E \longrightarrow \mathcal{Q}(^m F_1;E)~,~\delta(x) = id ^m_{\mathbb{K}}\otimes x,$$
$$\gamma \colon \mathcal{Q}(^m E_1;F) \longrightarrow F~,~ \gamma(P) = P(1).$$
Then $R=\gamma\circ S_R\circ\delta$, and, since $S_R$ belongs to $\cal A$, it follows that $R\in \mathcal{A}(E;F)$.

Finally, assume that the operator ideal $\cal A$ and the Banach polynomial ideal $\left(\mathcal{Q},\|\cdot\|_{\mathcal{Q}}\right)$ are injective. Let $R\in \mathcal{L}(E;F)$ and let $I\colon F \longrightarrow G$ be a metric injection such that $I\circ R\in\mathcal{A}^{comp}_{m,left}(E;G)$. By definition, for all Banach spaces $E_1$, $F_1$ and any $B\in\Delta^1_m\mathcal{A}(E_1;F_1)$, we have $S_{I\circ R}\in\mathcal{A}\left( \mathcal{Q}(^m F_1;E);\mathcal{Q}(^m E_1;G)\right)$. By the injectivity of $\left(\mathcal{Q},\|\cdot\|_{\mathcal{Q}}\right)$ we know that the operator   $$\delta_I\colon \mathcal{Q}(^m E_1; F) \longrightarrow \mathcal{Q}(^m E_1;G)~,~\delta_I(P) = I\circ P,$$
is metric injection, hence the injectivity of $\mathcal{A}$ and $S_{I \circ R} = \delta_I \circ S_R$ yield that $S_R$ belongs to $ \mathcal{A}$, and therefore $R\in\mathcal{A}^{comp}_{m,left}(E;F)$.
\end{proof}

\bigskip

\noindent Faculdade de Matem\'atica~~~~~~~~~~~~~~~~~~~~~~Departamento de Matem\'atica\\
Universidade Federal de Uberl\^andia~~~~~~~~ IMECC-UNICAMP\\
38.400-902 -- Uberl\^andia -- Brazil~~~~~~~~~~~~ 13.083-859 - Campinas -- Brazil\\
e-mail: botelho@ufu.br ~~~~~~~~~~~~~~~~~~~~~~~~~e-mail: leodan.ac.t@gmail.com

\end{document}